\documentclass[12pt]{article}

\usepackage[skins]{tcolorbox}

\usepackage[symbol]{footmisc}
\usepackage{bm}
\usepackage{amsmath}
\usepackage{amssymb}
\usepackage{amsthm}
\usepackage{graphicx}
\usepackage{relsize}
\usepackage{hyperref}
\usepackage{enumerate}
\usepackage{mathtools}

\newcommand{\ba}{\backslash}

\newcommand{\Z}{\mathbb{Z}}

\usepackage[font=small,labelfont=bf]{caption} % Required for specifying captions to tables and figures

\newtheorem{theorem}{Theorem}

\begin{document}

\begin{center}
{\bf New Strongly Regular Graphs Found via Local Search for Partial Difference Sets}  \\
Andrew C. Brady\footnote[1]{University of Richmond, VA, acbrady2020@gmail.com.} \\
\today \\ %TODO change date as needed
\end{center}

{\bf Abstract} {\it Strongly regular graphs (SRGs) are highly symmetric combinatorial objects, with connections to many areas of mathematics including finite fields, finite geometries, and number theory. One can construct an SRG via the Cayley Graph of a regular partial difference set (PDS). Local search is a common class of search algorithm that iteratively adjusts a state to (locally) minimize an error function. In this work, we use local search to find PDSs. We found PDSs with 62 different parameter values in 1254 nonisomorphic groups of orders at most 147. Many of these PDSs replicate known results. In two cases, (144,52,16,20) and (147,66,25,33), the PDSs found give the first known construction of SRGs with these parameters. In some other cases, the SRG was already known but a PDS in that group was unknown. This work also corroborates the existence of (64,18,2,6) PDSs in precisely 73 groups of order 64.}


%61 is up to 144, 62 is with the 147 case 

\section{Introduction}

A graph is {\it strongly regular} if every vertex has degree $k$, every pair of adjacent vertices has $\lambda$ common neighbors, and every pair of non-adjacent vertices has $\mu$ common neighbors, for some $k,\lambda,\mu \in \Z$. Some strongly regular graphs (SRGs) are constructed from finite fields or finite geometries \cite{SRGbook}. 

In this paper, we require that the SRG has a regular automorphism group. Let $G$ be a group under multiplication and $1 \in G$ be the identity of the group. A subset $D \subset G$ is a {\it partial difference set} (PDS) if the set of differences $\{d_1 d_2^{-1} | d_1,d_2 \in D\}$ contains $\lambda$ copies of every nonidentity element in $D$ and $\mu$ copies of every nonidentity element not in $D$. PDSs are generalizations of difference sets and can be used to create projective 2-weight error correcting codes \cite{calderbankkantor}. Ma (1994) gives a nice survey of PDSs and how they relate to many other geometric and combinatorial objects \cite{MaSurvey}. Defining $D^{(-1)} := \{d^{-1} | d \in D\}$, a {\it regular} PDS is a subset $D \subset G$ such that D is a PDS, $1 \not\in D$, and $D = D^{(-1)}$. In this work, we require that our PDSs be regular, and henceforth refer to regular PDSs as simply PDSs. Partial difference sets have the additional interesting property that their Cayley Graph is strongly regular \cite{AbelTech}. Thus, by finding PDSs, we get an SRG for free. However, PDSs are difficult to find because the search space is so large: naively, to find all PDSs with parameters $(n,k,\lambda,\mu)$ in a group $G$ of order $n$ would require checking all ${n-1 \choose k}$ sets of size $k$. PDSs in nonabelian groups are especially difficult to find, because many mathematical techniques for finding PDSs function best in abelian groups. We attack the problem of finding PDSs in any group using local search. Because local search is such a flexible method, the techniques used in this paper could be adapted to search for other types of combinatorial objects. 

Local search is a class of search techniques that iteratively change a state in order to minimize an error function. Metaphorically, if the error function is a snowy mountain range, local search is choosing some starting location (a state), then sledding down a hill (iterative improvement) to a valley (local minimum). Advantages of local search include low space usage (because only the current state is stored) and fast search speed. However, because local search uses iterative improvement, local search can only find local minima of the error function, so there is no guarantee of solution optimality. For more information on local search, see chapter 4 of Russell and Norvig's classic textbook \cite{AIbook}.

The main result of this paper is the following theorem:

\begin{theorem} There exist (144,52,16,20) and (147,66,25,33) SRGs.
\end{theorem}

\begin{proof}
We show existence by explicit construction. In \ref{newSRG} we give PDSs with the desired parameters. These PDSs are verifiable in GAP, C++, or by hand. As the Cayley Graph of a PDS is strongly regular, we have constructed the desired SRGs. 
\end{proof}

This paper is organized as follows. In Section \ref{RW}, we discuss other approaches to finding PDSs. In Section \ref{M}, we elaborate on the details of the local search. In Section \ref{R}, we discuss the new PDSs found. In Section \ref{CAFW}, we discuss limitations and future directions of this work. Tables are displayed at the end of the paper. 
 
\section{Related Work \label{RW}} 

To find new PDSs, one can rely on mathematical constructions or computational search. Originally, most constructions of SRGs from PDSs came as consequences of results from finite fields, finite geometries, projective codes, and other mathematical structures \cite{MaSurvey}. The underlying group in this case would often be abelian. More recently, Polhill, Davis, Smith, and Swartz (2023) constructed genuinely nonabelian PDS families, which are PDS families in nonabelian groups without a corresponding family in abelian groups \cite{GenNon}. They also recently (2023) constructed PDS families in nonabelian groups by adapting techniques used to construct PDSs in abelian groups \cite{AbelTech}. 

On the computational side, for example, AbuGhneim, Peifer, and Smith (2019) used computer search to construct all $(96,20,4,4)$ and $(96,19,2,4)$ PDSs \cite{all96}. Peifer (2019) constructed difference sets in many small groups \cite{DifSets}. Jørgenson and Klin (2003) found various PDSs in groups of order 100 with the aid of computational tools \cite{non100}. Brady (2022) exhaustively searched many groups of order 64 for Negative Latin Square Type PDSs with dimensions $(64,18,2,6)$ \cite{me}. Previous computational searches have typically implemented some theoretical mathematical restriction in order to reduce the search space to find results exhaustively. The results in this paper are unusual in that no restrictions are imposed on the target group: thus, this work is far more flexible than prior research. To the author's knowledge, local search has never been used in published research to look for PDSs.  

\section{Methods \label{M}}

A {\it state} in our local search will be a single set $D$ of $k$ elements, and a {\it move} will consist of replacing one element in $D$ (one swap). The local search will be a basic hill climb, that moves to its first seen neighbor with lower error. A single {\it trial} consists of choosing 1 initial starting state and iteratively updating the state until convergence. To choose a swap, a random element in $G \ba \{1\}$ is chosen to replace a random element in $D$ (for $(n-1)k$ total possible swaps); however, the swap is automatically invalid if the swap inserts an element already in $D$. We define {\it convergence} to be $\alpha$ swaps in a row being invalid or failing to improve $D$, for some fixed $\alpha$ we choose. Though most of our runs used $\alpha=(n-1)k-1$, we recommend setting $\alpha=(n-1)k$, so that all neighbors of a state are checked before the search gives up on a trial.

In the group ring $\Z[G]$, from our definition of PDS, it follows that a subset $D$ of a group $G$ is a PDS with parameters $(n:=|G|,k:=|D|,\lambda,\mu)$ if and only if we have that  \begin{equation*}
D := \sum_{d \in D} d, G := \sum_{g \in G} g, D^2 = k(1) + \lambda D + \mu (G \ba (\{ 1 \} \cup D)). \end{equation*} 

For an element $\sum_{g \in G} c_g g \in \Z[G]$, define $l_2(\sum_{g \in G} c_g g ) := \sum _{g \in G} |c_g|^2$ and define $\textrm{Diff}(D) := D^2 - k(1) - \lambda D - \mu (G \ba (\{ 1 \} \cup D))$. We define our error function to be \begin{equation*} e(D) := l_2(\textrm{Diff}(D)). \end{equation*} Note that a set $D$ is an PDS iff $e(D)=0$. Also note that computing e(D) without prior knowledge of $D^2$ takes $O(k^2+n)$ time, whereas updating e(D) after changing a single element of D takes $O(k)$ time. Thus choosing a swap as our state transition significantly speeds up our local search. 

The GAP package GRAPE was used to generate Cayley graphs and check graph isomorphism \cite{GRAPE}, and the GAP package Digraphs was used to check whether a graph came from a partial geometry \cite{Digraphs}. We used the GroupNames database as a reference to learn about the structure of various groups \cite{GroupNames}. The convolution tables (also known as multiplication or group tables) for the groups searched were generated with GAP \cite{GAP}. For small group generation and indexing we used the SmallGrp GAP package \cite{SmallGrp1.4.1}. These tables were then fed into C++, where the search was implemented. The searches were performed in C++ because C++ is a fast language. Randomness was generated using std::mt19937, a Mersenne Twister. 

Local search can be easily parallelized because it consists of many independent tasks. If many groups and parameter sets are being searched, as in \ref{newPDSsec}, one can run different groups on separate processors. If only one group is being searched, as in \ref{PDS512Sec}, then independent trials can be run on separate processors. The parallelized code was written with parlay, a C++ library for writing parallel code \cite{Parlay}. Parlay was chosen for its simplicity and effectiveness. 

Runs were performed on two machines: the author's laptop (Macbook Pro 2019, 32GB Memory, 2.6 GHz 6-Core Intel Core i7, Sonoma) and on Spydur, University of Richmond's HPC environment \cite{Spydur}. On Spydur, runs were performed on ``basic'' nodes, which each contain 52 Xeon cores and 384GB of memory. Runs were scheduled with SLURM; sequential runs were given 1 core and parallel runs were typically given 40 cores. 

\section{Results \label{R}}

\subsection{Recall of Local Search}

What is the recall of local search using $N$ trials? That is, how many groups with a known PDS will yield at least 1 PDS when doing $N$ trials of local search? Dr. Ken Smith, in currently unpublished work, found exhaustively, using a very different method, that precisely 73 groups of the 267 nonisomorphic groups of order 64 had (64,18,2,6) PDSs. As a proof of concept, we repeated this search using local search with 10,000 trials. We found a PDS in 73 groups, meaning that we got 100\% recall in this case. These results agree with Brady (2022) \cite{me}, which exhaustively found that, among groups with SmallGroup Id 55-266 of order 64, exactly 49 had (64,18,2,6) PDSs. Thus, even though local search is by nature not exhaustive, in small cases it can suggest nonexistence of some PDSs. 

\subsection{New PDSs Found \label{newPDSsec} }

All SRG dimensions with $n \le 144$ were searched for PDSs in every group, as well as SRG dimensions with $145 \le n \le 215$ or $218 \le n \le 238$, where the existence of an SRG was unknown according to Brouwer's table \cite{Brouwer}. See Table \ref{trialstab} for the maximum number of trials of local search run for each parameter set. For each parameter set and group, trials were run until a PDS was found or the maximum number of trials ran without finding a PDS. Thus, the vast majority of trials were spent searching for PDSs in groups where they likely do not exist. For example, only 130 trials were needed to generate a (147,66,25,33) PDS in SmallGroup(147,4), but 43218 trials were spent failing to find a (147,66,25,33) PDS in SmallGroup(147,5).

\begin{table}
\begin{center}
\begin{tabular}{c|c}
Group size & \# Trials \\
\hline
$n < 144$ & $5n^2$ \\
$n=144, k < 34$ & $n^2$ \\
$n=144, k \ge 34$ & $n^2,2n^2,2n^2$ \\
$145 \le n < 162$, SRG unknown & $2n^2$ \\
$162 \le n < 186$, SRG unknown & $2n^2,2n^2$ \\
$186 \le n < 239, n \not\in \{216,217\}$, SRG unknown & $2n^2$ \\
\hline
\end{tabular}
\end{center}
\caption{Maximum number of trials of local search ran on each parameter set. The number of trials was chosen to scale based on the group order, because larger groups have a larger search space. In the $n=144,k \ge 34$ case, for example, three searches were run: one with $n^2$ trials, and two with $2n^2$ trials.}
\label{trialstab}

\end{table}

In total, PDSs with 62 different parameter values were found in 1254 nonisomorphic groups. Many of these PDSs replicate known results, but some are new. We highlight two parameter sets for which we found a new SRG in Table \ref{newSRG}. To extract the SRGs from the PDSs given in Table \ref{newSRG}, adapt GAP code in Figure \ref{pdsCode}. We found two (144,52,16,20) PDSs, one in SmallGroup(144,68) and one in SmallGroup(144,126). The Cayley graphs of these PDSs give isomorphic (144,52,16,20) SRGs. From De Winter, Kamischke, and Wang we know that the (144,52,16,20) PDSs are genuinely nonabelian \cite{AutoSRG}. We found two (147,66,25,33) PDSs, one in SmallGroup(147,3) and one in SmallGroup(147,4). The Cayley graphs of these PDSs give isomorphic (147,66,25,33) SRGs. From Theorem 3.4 in Ma's PDS survey \cite{MaSurvey}, we know that the (147,66,25,33) PDSs are genuinely nonabelian. Note that the (147,66,25,33) SRG does not yield the partial geometry pg(6,10,3); although the clique number of this graph is 7, there are too few cliques of size 7 to give a partial geometry \cite{BrouwerAnalysis}. From Brouwer we know that the existence of (144,52,16,20) and (147,66,25,33) SRGs was open \cite{Brouwer}. 

\begin{table}[h!]

\begin{center}
\begin{tabular}{c|p{12cm}}
{\bf (144,52,16,20)} & SmallGroup(144,68): [2, 3, 5, 8, 10, 14, 17, 18, 22, 23, 32, 33, 43, 44, 45, 46, 47, 50, 54, 56, 57, 60, 61, 63, 68, 70, 71, 73, 75, 77, 79, 80, 81, 82, 85, 87, 91, 92, 94, 100, 105, 106, 108, 111, 113, 117, 126, 139, 140, 141, 142, 144] \\
& SmallGroup(144,126): [2, 3, 5, 8, 9, 15, 16, 19, 21, 22, 23, 25, 29, 30, 31, 35, 39, 40, 44, 45, 47, 48, 51, 52, 53, 54, 55, 56, 60, 67, 71, 74, 77, 86, 89, 90, 97, 98, 101, 103, 108, 111, 112, 118, 121, 127, 128, 136, 139, 140, 142, 143] \\
\hline
{\bf (147, 66, 25, 33)} & SmallGroup(147,3): [2, 3, 4, 5, 8, 10, 11, 12, 14, 16, 19, 20, 25, 27, 28, 29, 30, 31, 33, 37, 38, 39, 40, 41, 42, 45, 46, 52, 53, 55, 57, 58, 64, 66, 67, 70, 71, 79, 82, 87, 88, 89, 93, 94, 103, 104, 105, 107, 108, 111, 112, 113, 116, 117, 126, 127, 128, 129, 130, 131, 133, 134, 135, 142, 145, 147] \\
& SmallGroup(147,4): [3, 4, 6, 7, 8, 9, 11, 15, 16, 18, 20, 22, 23, 25, 26, 27, 32, 33, 35, 41, 42, 43, 45, 49, 50, 51, 52, 54, 56, 58, 62, 63, 64, 66, 72, 74, 77, 80, 81, 85, 86, 88, 89, 90, 91, 97, 98, 104, 111, 112, 113, 115, 116, 120, 121, 122, 123, 124, 125, 131, 132, 133, 136, 138, 141, 143] \\
\hline
\end{tabular}
\end{center}
\caption{New SRGs found and their associated PDSs. Please note that in total one (144,52,16,20) and one (147,66,25,33) SRG were found; the Cayley graphs of the PDSs give isomorphic SRGs. To extract the SRGs, see GAP code in Figure \ref{pdsCode}. }
\label{newSRG}
\end{table}

\begin{figure}
\begin{verbatim} 
g := SmallGroup(144,68); #Define g to be group of interest
e := Elements(g); 
pds := [2,3,5,8,10,...,142,144]; #Take from Table 1
LoadPackage("grape");
graph := CayleyGraph(g,List(pds,x->e[x]));
\end{verbatim} 
\caption{Code to extract an SRG from a PDS in Table \ref{newSRG}.}
\label{pdsCode}
\end{figure}

Please refer to the Github (\url{https://github.com/42ABC/PDS_local/}) for the complete list of PDSs, which includes PDSs with parameters (81,32,13,12), (81,40,19,20), (125,28,3,7), (125,52,15,26), (125,62,30,31), (144,55,22,20), and (144,66,30,30). Please note that PDSs in the source code are zero-indexed (C++ style), whereas PDSs included in the paper are 1-indexed (GAP-style). Thus, to load any PDS in the repository into GAP, 1 must be added to every element. 

\subsection{PDS found in EA(512) \label{PDS512Sec} }

Running $111100$ trials, three (512,70,6,10) PDSs were found in the elementary abelian group of order 512. Although these PDSs were already known, the fact that local search could even find a PDS of size 512 demonstrates that local search can function on larger group sizes. However, the fact that so much many trials were needed to find a PDS in EA(512), arguably the easiest group of order 512 to search, suggests that better versions of local search are needed to handle this scale. 

\section{Conclusions, Future Work, and Acknowledgements \label{CAFW}}

In this project, we used local search to find PDSs in many small groups, some of which are new. Two of the new PDSs are used to construct SRGs listed as open in Brouwer's tables \cite{Brouwer}. Additionally, these PDSs will serve as useful large examples to researchers, to help them understand the structure of PDSs. 

This work has two main limitations. First, because local search is not exhaustive, we cannot use this method to determine nonexistence of a PDS in a particular group. Consequently, determining when to stop a run can be difficult, because running more trials will not yield a PDS if a PDS does not exist in that group. Second, the PDSs found are necessarily a mix of new and already known PDSs. It is difficult to know which are new and which are not because there is not an equivalent of Brouwer's table \cite{Brouwer} for PDSs. 

The SRGs provided in this paper could help construct other novel SRGs. For instance, Brouwer recently used graph switching on the (147,66,25,33) SRG to construct a (148,77,36,44) SRG, the existence of which was previously unknown \cite{BrouwerAnalysis}. 

This paper managed to find interesting results simply by computer search, without incorporating mathematical knowledge into the search. Thus, if mathematical knowledge was incorporated, the search could become even more effective. One way to do this would be to do local search on a search space already reduced by character theory. Another way would be to do local search with a starting set mathematically close to a PDS. Another idea is to fix a set of linking difference sets and do local search on the coset representatives attached to these difference sets, in an effort to find a large PDS, since certain PDSs can be broken down into difference sets. A simple modification would be to force inverses to be swapped in/out together (because $D=D^{(-1)}$). 

It is worth highlighting that local search is a highly flexible method. Thus, the ideas in this paper can be easily adapted to find difference sets, relative difference sets, partial difference sets in general (instead of PDSs), or any other combinatorial object where an error is easily defined. 

\noindent \textbf{Acknowledgements.} Thank you to Drs. James Davis, Eric Swartz, Ken Smith, and John Polhill for contributions to the work and helping to edit the paper draft. Thank you to Drs. Ken Smith and Ryan Kaliszewski for the incidence matrix and convolution table code used in this work. Thank you to the University of Richmond for the computing resources (Spydur) on which runs were performed and to George Flanagin for information about Spydur and SLURM. Thank you to Andries Brouwer, John Bamberg, and Ferdinand Ihringer for helpful comments following the first submission of the preprint. This paper started as a final project for Dr. Joonsuk Park's AI course; I would like to thank Dr. Park for making the assignment so open-ended that I could have the time to pursue this interesting line of research. 

\noindent \textbf{Declaration of Interest.} The author reports there are no competing interests to declare. 

\bibliographystyle{plain} 
\bibliography{my_bib}{}

% (for pointing out that the (147,66,25,33) SRG did not yield a pg(6,10,3) and explaining how to determine whether a graph is induced by a partial geometry),

\end{document}